\documentclass[12pt]{amsart}

\title[Hyperconvex manifold without bounded holomorphic functions]{On a hyperconvex manifold without non-constant bounded holomorphic functions}

\author{Masanori Adachi}
\address{Department of Mathematics, Faculty of Science, Shizuoka University.  836 Ohya, Suruga-ku, Shizuoka 422-8529, Japan.}
\email{adachi.masanori@shizuoka.ac.jp}

\date{\today}

\dedicatory{Dedicated to Professor~Kang-Tae~Kim on the~occasion of his~60th ~birthday} 

\usepackage{amsmath,amsthm,amssymb,latexsym}
\usepackage[abbrev]{amsrefs}

\newtheorem*{Theorem*}{Theorem}
\newtheorem*{Conjecture*}{Conjecture}

\newtheorem{Theorem}{Theorem}[section]
\newtheorem{Fact}[Theorem]{Fact}
\newtheorem{Proposition}[Theorem]{Proposition}
\newtheorem{Definition}[Theorem]{Definition}

\newtheorem{Problem}{Problem}

\theoremstyle{remark}
\newtheorem{Remark}[Theorem]{Remark}

\newcommand\C{\mathbb{C}}  
\newcommand\R{\mathbb{R}}

\newcommand\PP{\mathbb{P}}

\newcommand{\pa}{\partial}
\newcommand{\opa}{\overline\pa}
\newcommand{\ol}{\overline }
\newcommand{\D}{\mathbb{D}}

\begin{document}

\maketitle

\begin{abstract}
An example is given of a hyperconvex manifold without non-constant bounded holomorphic functions, 
which is realized as a domain with real-analytic Levi-flat boundary in a projective surface.
\end{abstract}

\section{Introduction}
In geometric complex analysis, hyperbolicity and parabolicity of non-compact complex manifolds 
are key properties governing behavior of holomorphic functions.
Stoll \cite{Sto} introduced the notion of \emph{parabolic manifold} to investigate value distribution of holomorphic functions in several variables. We recall this notion using the formulation of Aytuna--Sadullaev \cite{AySa}:
\begin{Definition}
A complex manifold $X$ is said to be \emph{parabolic} if $X$ does not admit non-constant bounded plurisubharmonic function. We say that $X$ is \emph{$S$-parabolic} if it possesses a plurisubharmonic exhaustion $\varphi$ that satisfies the homogeneous complex Monge--Amp\`ere equation $(i\pa\opa \varphi)^n = 0$ on $X \setminus K$ for some compact subset $K \subset X$. 
\end{Definition}
$S$-parabolic manifolds are parabolic, and their model case is $\C^n$ equipped with the exhaustion $\log \|z\|$.
We refer the reader to Aytuna--Sadullaev \cite{AySa} for the detail.

On the other hand, it would also be of  interest to investigate non-compact complex manifolds that are not parabolic in the sense above but enjoy some weaker parabolicity. 
Myrberg \cite{My} gave such an example in one dimensional setting, namely, an open Riemann surface of infinite genus 
that has smooth boundary component, hence, not parabolic, but on which all the bounded holomorphic functions are constant. 
This celebrated example was the driving force toward the classification theory of Riemann surfaces (cf. Heins \cite{He}). 

In several complex variables, this sort of intermediate parabolicity actually appears too. See  Aytuna--Sadullaev \cite{AySa} for an example of unbounded pseudoconvex domain in $\C^n$ containing countably many copies of $\C$ and having plurisubharmonic defining function but no bounded holomorphic function except for constant functions. 
The purpose of this article is to remark another kind of example of non-parabolic Stein manifold without non-constant bounded holomorphic function, which the author hopes to be useful for further study. 

\begin{Theorem*}
There exists a hyperconvex manifold that does not possess any non-constant bounded holomorphic function and is realized as a domain with real-analytic Levi-flat boundary. 
\end{Theorem*}

Here hyperconvexity is defined as

\begin{Definition}
A complex manifold $X$ is said to be \emph{hyperconvex} if it admits strictly plurisubharmonic bounded exhaustion. 
\end{Definition}

Recall that a function on a topological space $X$, $\varphi\colon X \to [-\infty, c)$, is said to be \emph{bounded exhaustion} if all the sublevel sets $\{ x \in X \mid \varphi(x) < b \}$, $b <c$, are relatively compact in $X$. For example, any $C^2$-smoothly bounded pseudoconvex domains in Stein manifolds is hyperconvex (Diederich--Forn{\ae}ss \cite{DiFo}. Later the required smoothness was relaxed to $C^1$ by Kerzman--Rosay \cite{KeR}, then to Lipschitz boundary by Demailly \cite{De}). Clearly, a hyperconvex manifold is not parabolic, but Theorem states that it can satisfy the Liouville property. 

Now we explain the construction of the manifold claimed in Theorem. Let $\Sigma$ be a compact Riemann surface of genus $\geq 2$ and fix its uniformization $\Sigma = \D / \Gamma$ by a Fuchsian group $\Gamma$ acting on the unit disk $\D$. We make $\Gamma$ act on the bidisk $\D \times \D$ diagonally but with conjugated complex structure for second factor, namely, for each $\gamma \in \Gamma$  and $(z, w) \in \D \times \D$, we let
\[
\gamma \cdot (z, w) := (\gamma z, \overline{\gamma \overline{w}}).
\]
We shall show that the quotient space $X := \D \times \D / \Gamma$ enjoys the desired property. 

This example has two origins. One is the work by Diederich--Ohsawa \cite{DiO}, where holomorphic $\D$-bundles over compact K\"ahler manifolds are shown to be weakly 1-complete. Such a holomorphic $\D$-bundle is canonically embedded in the associated holomorphic $\C\PP^1$-bundle as a pseudoconvex domain with real-analytic Levi-flat boundary.
In our case, the first and the second projection endow $X$ structures of $\D$-bundle over $\Sigma$ and $\ol{\Sigma}$, the quotient of $\D$ by the conjugated action of $\Gamma$, respectively. Hence, $X$ has two realization as domains in ruled surfaces $Y := \D \times \C\PP^1 / \Gamma$ and $Y' := \C\PP^1 \times \D /\Gamma$, where the action of $\Gamma$ is the same as above thanks to the fact $\mathrm{Aut}(\D) \subset \mathrm{Aut}(\C\PP^1)$. The Levi-flat boundaries of $X$ in $Y$ and $Y'$ are denoted by $M = \D \times \pa\D / \Gamma$ and $M' = \pa\D \times \D /\Gamma$ respectively. In summary, we have two natural ways to realize $X$ in larger complex manifolds $Y$ and $Y'$ and the real-analytic boundaries $M$ and $M'$ are inequivalent CR manifolds in general (Mitsumatsu \cite{Mi}). 
For further background on $\D$-bundles, we refer the reader to a recent study by Deng--Forn{\ae}ss \cite{DeFo}.

Another origin is the Grauert tube of maximal radius in the sense of Guillemin--Stenzel \cite{GuSt} and Lempert--Sz\H{o}ke \cite{LSz}. Since the conjugated diagonal set $\{ (z, \overline{z}) \mid z \in \D\} \subset \D \times \D$ is preserved under the action of $\Gamma$, its quotient $S$ is totally-real submanifold of real dimension two and isomorphic to $\Sigma$ as real-analytic manifold. Namely, $X$ is a complexification of $\Sigma$. Not only that, we can find a plurisubharmonic \emph{bounded} exhaustion that satisfies the homogeneous complex Monge--Amp\`ere equation on $X \setminus S$.

In \S2, we first confirm that our $X$ coincides with the Grauert tube of $\Sigma$, then show the hyperconvexity of $X$.
In \S3, after explaining that the Liouville property of $X$ is actually a corollary of Hopf's ergodicity theorem, we shall give another proof for the Liouville property using the plurisubharmonic bounded exhaustion. In \S4, some open questions are posed. 

\section{Grauert tube and its hyperconvexity}

First we recall the notion of Grauert tube  in the sense of Guillemin--Stenzel  and Lempert--Sz\H{o}ke. 
\begin{Fact}[Guillemin--Stenzel \cite{GuSt}, Lempert--Sz\H{o}ke \cite{LSz}]
\label{GSLS}
Let $(M, g)$ be a compact real-analytic Riemannian manifold of dimension $n$. Denote by $\rho\colon TM \to \R_{\geq 0}$ the length function, and we identify $M$ with the zero section of $TM$. Then, there exists $R \in (0, \infty]$ and unique complex structure on $X := \{ v \in TM \mid \rho(v) < R \}$ such that \begin{enumerate}
\item $\rho$ enjoys the homogeneous complex Monge--Amp\`ere equation $(i\pa\opa \rho)^n = 0$ on $X \setminus M$;
\item $\rho^2$ is strictly plurisubharmonic on $X$;
\item $i\pa\opa (\rho^2)$ agrees with $g$ on $TM$.
\end{enumerate}
\end{Fact}
This $X$ above is called \emph{the Grauert tube of $\Sigma$ of radius $R$}. 
Since our $\Sigma$ is endowed with the hyperbolic metric of constant Gaussian curvature $-1$, whose fundamental form is
\[
g(z) = \frac{2i dz \wedge d\ol{z}}{(1-|z|^2)^2},
\]
Lempert--Sz\H{o}ke \cite[Theorem 4.3]{LSz} yields an upper bound of the radius $R$ of the Grauert tube of $\Sigma$, $R \leq \pi/2$. 

\begin{Proposition}
\label{grauert_tube}
The complex manifold $X$ defined in \S1 is biholomorphic to the Grauert tube of $\Sigma$ of radius $\pi/2$, which is maximum possible, whose length function agrees with
\[
\rho(z,w) := \arccos \sqrt \delta \quad 
\text{where} \quad 
\delta(z, w) := 1 - \left| \frac{w - \overline{z}}{1 - zw} \right|^2.
\]
\end{Proposition}

\begin{proof}
First note that $\delta \colon \D \times \D \to (0,1]$ is invariant under the action of $\Gamma$ and induces a real-analytic function on $X$. 
Hence, $\rho \colon X \to [0, \pi/2)$ is well-defined bounded exhaustion and $\rho^{-1}(0) = S = \{ (z, \overline{z}) \mid z \in \D\}/\Gamma$, 
which we identified with $\Sigma$. Moreover, $\rho^2$ is $C^\infty$-smooth function on $X$ since
\[
\rho(z, w) = \arcsin \left| \frac{w - \overline{z}}{1 - zw} \right|.
\]

In view of Lempert--Sz\H{o}ke \cite[Theorem 3.1]{LSz}, it suffices to confirm that $\rho$ satisfies the three conditions in Fact \ref{GSLS}.
From direct computation, we have
\begin{align*}
i\pa\opa (-\log \delta) &= \frac{i dz \wedge d\ol{z}}{(1-|z|^2)^2} +  \frac{i dw \wedge d\ol{w}}{(1-|w|^2)^2},\\
\frac{i\pa (-\log \delta) \wedge \opa (-\log \delta)}{ 1-\delta } &= \frac{i dz \wedge d\ol{z}}{(1-|z|^2)^2} +  \frac{i dw \wedge d\ol{w}}{(1-|w|^2)^2}+ \frac{i \varepsilon dz \wedge d\ol{w} + i  \ol{\varepsilon} dw \wedge d\ol{z}}{(1-|z|^2)(1-|w|^2)}
\end{align*}
on $X \setminus S$, where $\varepsilon = -(w-\ol{z})(\ol{w} - {z})^{-1}$.  Hence, it follows that
\begin{align*}
\opa \rho 
&=  \frac{1}{2} \sqrt{\frac{\delta}{1 - \delta}} \opa (- \log \delta),\\
i\pa\opa \rho 
&= \frac{1}{2} \sqrt{\frac{\delta}{1-\delta}}\left( i\pa\opa (-\log \delta) - \frac{1}{2} \frac{ i\pa (-\log \delta) \wedge \opa (-\log \delta)}{ 1-\delta } \right)\\
&= \frac{1}{4} \sqrt{\frac{\delta}{1-\delta}}\left( \frac{i dz \wedge d\ol{z}}{(1-|z|^2)^2} +  \frac{i dw \wedge d\ol{w}}{(1-|w|^2)^2} - \frac{i \varepsilon dz \wedge d\ol{w} + i  \ol{\varepsilon} dw \wedge d\ol{z}}{(1-|z|^2)(1-|w|^2)}
 \right),
\end{align*}
and it is now clear that $(i\pa\opa\rho)^2 = 0$ on $X \setminus S$. To check remaining two points, we compute on $X \setminus S$ 
\begin{align*}
i\pa\opa(\rho^2)  &=  2 (\rho i\pa\opa\rho + i\pa \rho \wedge \opa \rho) \\
&= \frac{1}{2} \left( \rho \sqrt{\frac{\delta}{1-\delta}} +  \delta \right) \left(\frac{i dz \wedge d\ol{z}}{(1-|z|^2)^2} +  \frac{i dw \wedge d\ol{w}}{(1-|w|^2)^2}\right)\\
&\quad + \frac{1}{2} \left( -\rho \sqrt{\frac{\delta}{1-\delta}} + \delta\right) \frac{i \varepsilon dz \wedge d\ol{w} + i  \ol{\varepsilon} dw \wedge d\ol{z}}{(1-|z|^2)(1-|w|^2)}.
\end{align*}
It follows that $i\pa\opa(\rho^2) > 0$ on $X$, and $g$ agrees with the restriction of 
\[
i\pa\opa(\rho^2) = \frac{i dz \wedge d\ol{z}}{(1-|z|^2)^2} +  \frac{i dw \wedge d\ol{w}}{(1-|w|^2)^2} = \frac{i dz \wedge d\ol{z} + dw \wedge d\ol{w}}{(1-|z|^2)^2} 
\]
on $S$ as Riemannian metric. The proof is completed.
\end{proof}

\begin{Remark}
Kan \cite{Ka} gave another realization of the Grauert tube of $\Sigma$ extending the construction of Lempert \cite{L}.
\end{Remark}

Next we shall confirm that our $X$ is hyperconvex. 
\begin{Proposition}
\label{hyperconvex}
The function $-\sqrt{\delta}$ is strictly plurisubharmonic bounded exhaustion on $X$. Hence, $X$ is hyperconvex.
\end{Proposition}

\begin{proof}
From the computation in the proof of Proposition \ref{grauert_tube}, we have
\begin{align*}
\frac{i\pa\opa (-\sqrt{\delta})}{\sqrt{\delta}/2} &= i\pa\opa(-\log \delta) -\frac{1}{2}i\pa(-\log \delta)\wedge \opa(-\log \delta)\\
&= \frac{1+\delta}{2} \left(\frac{i dz \wedge d\ol{z}}{(1-|z|^2)^2} +  \frac{i dw \wedge d\ol{w}}{(1-|w|^2)^2}\right)+ \frac{1-\delta}{2} \frac{i \varepsilon dz \wedge d\ol{w} + i  \ol{\varepsilon} dw \wedge d\ol{z}}{(1-|z|^2)(1-|w|^2)}
\end{align*}
and this is positive definite everywhere on $X$. 
\end{proof}

\begin{Remark}
We may extend $\delta$ smoothly on a neighborhood of $X$ in $Y$ and also a neighborhood in $Y'$ and regard $-\delta$ as a defining function of $X$ in $Y$ and $X$ in $Y'$. 
Proposition \ref{hyperconvex} shows, by its definition, that $-\delta$ has the Diederich--Forn{\ae}ss exponent $1/2$, which is the maximum possible value for 
relatively compact domains with Levi-flat boundary in complex surfaces (Fu--Shaw \cite{FuSh} and Adachi--Brinkschulte \cite{AdB1}.
See also Demailly \cite[Th\'eor\`eme 6.2]{De}).
\end{Remark}

\section{Proofs of the Liouville property}

Let us observe that the Liouville property of $X$ is actually a corollary of Hopf's ergodicity theorem (\cite{Ho}. See also Tsuji \cite{T}, Garnett \cite{Ga} and Sullivan \cite{Su}). 

\begin{Fact}[Hopf \cite{Ho}]
\label{Hopf}
Let $\Sigma = \D/\Gamma$ be a Riemann surface of finite hyperbolic area.
Then, the diagonal action of $\Gamma$ on $\pa\D \times \pa\D$ is ergodic with respect to its Lebesgue measure. Namely, for any Lebesgue measurable subset $E \subset \pa\D \times \pa\D$ invariant under the diagonal action of $\Gamma$ has Lebesgue measure zero or full Lebesgue measure. 
\end{Fact}

We use the following Fatou type theorem.
\begin{Fact}[cf. Tsuji {\cite[Theorem IV.13]{T}} ]
\label{Fatou}
Let $f$ be a bounded holomorphic function on $\D \times  \D$. Then, 
there exists a measurable function $\tilde{f}\colon \pa\D \times \pa\D \to \C$ such that for almost all $(z_0, w_0) \in \pa\D \times \pa\D$,
\[
\lim_{(z, w) \to (z_0, w_0)} f(z,w) = f(z_0, w_0)
\]	
where $z$ and $w$ approach to $z_0$ and $w_0$ non-tangentially respectively.
Moreover, $f$ is a constant function if $\tilde{f}$ is constant on a subset of positive measure. 
\end{Fact}

\begin{Theorem}
\label{liouville}
Any bounded holomorphic function on $X$ is constant. 
\end{Theorem}

\begin{proof}[First proof of Theorem \ref{liouville}]
Let $f$ be a bounded holomorphic function on $X = \D \times \D /\Gamma$. 
From Fact \ref{Fatou}, $f$ as a function on $\D \times \D$ has boundary value $\tilde{f}$ on $\pa\D \times \pa\D$ which is invariant under the action of $\Gamma$. 
Then, the function $(z,w) \mapsto \tilde{f}(z, \ol{w})$ on $\pa\D \times \pa\D$ is invariant under the diagonal action of $\Gamma$. Fact \ref{Hopf} implies that $\tilde{f}$ is constant almost everywhere, and we conclude by Fact \ref{Fatou}.
\end{proof}

We shall give another proof, which does not rely on Fact \ref{Hopf} and explains how the bounded exhaustion $\rho$ controls the growth of holomorphic functions on $X$.

\begin{proof}[Second proof of Theorem \ref{liouville}]
Let $f$ be a bounded holomorphic function on $X$. We shall show without using Fact \ref{Hopf} that the boundary value function $\tilde{f}$ on $\pa\D \times \pa\D$ is constant almost everywhere. Then the rest of the proof is the same as in the first proof.

We apply the integration formula used in Adachi--Brinkschulte \cite{AdB2} with the maximal plurisubharmonic function $\rho$ on $X \setminus S$ used in Proposition \ref{grauert_tube}. Namely, we integrate
\[
i \pa\opa |f|^2 \wedge d \rho \wedge d^c \rho + |f|^2 (i \pa\opa \rho)^2 = d (d^c |f|^2 \wedge i\pa \rho \wedge \opa \rho + |f|^2 d^c \rho \wedge i \pa\opa\rho)
\]
on $\rho^{-1}(a, b)$, where our convention is $d^c := (\pa - \opa)/2i$. 
Since all the level sets $\rho^{-1}(c)$, $c \in (0, \pi/2]$, are smooth, for any $a, b \in (0, \pi/2)$, $a < b$, we have
\[
\int_{\rho^{-1}(a, b)} i \pa\opa |f|^2 \wedge d \rho \wedge d^c \rho = \int_{\rho^{-1}(b)} |f|^2 d^c \rho \wedge i \pa\opa\rho - \int_{\rho^{-1}(a)} |f|^2 d^c \rho \wedge i \pa\opa\rho.
\]
Denoting by $M_t$ the boundary of $\{ x \in X \mid \rho(x) < t \} = \{ x \in X \mid \delta(x) > \cos^2 t \}$ and rewriting in $\delta$ instead of $\rho$ yield
\begin{align}
\int_{\delta^{-1}(\beta, \alpha)} i \pa\opa |f|^2 \wedge \frac{d \delta \wedge d^c \delta}{\delta(1-\delta)} 
&= \frac{1}{\sin^2 b} \int_{M_\beta} |f|^2 d^c (-\delta) \wedge i \pa\opa(-\log \delta) 
\label{eq} \\
&\quad - \frac{1}{\sin^2 a} \int_{M_\alpha} |f|^2 d^c (-\delta) \wedge i \pa\opa(-\log \delta) \notag
\end{align}
where $\alpha := \cos^2 a$ and $\beta := \cos^2 b$. 

Now we look at behavior of terms in Equation (\ref{eq}) when $b \nearrow \pi/2$, that is, $\beta \searrow 0$. 
For its RHS, we compute the first term using a smooth trivialization
\[
\iota_t \colon R \times \pa\D \to M_t, \quad (z, e^{i\theta}) \mapsto \left(z, \frac{(\sin t)e^{i\theta}+\ol{z}}{1+z(\sin t)e^{i\theta}}\right) 
\]
for $t \in (0,\pi/2]$ where $R$ is a fundamental domain of the action of $\Gamma$ on $\D$. It follows that
\begin{align*}
& \frac{\beta}{\sin^2 b}  \int_{M_\beta} |f|^2 d^c (-\log \delta) \wedge i \pa\opa(-\log \delta) \\
&=\frac{\beta}{\sin^2 b} \int_{M_\beta} |f|^2 \left(\frac{i dz \wedge d\ol{z}\wedge \frac{1}{2i}\left( \frac{\ol{w}-z}{1-zw}dw - \frac{w-\ol{z}}{1-\ol{zw}}d\ol{w}\right)}{(1-|z|^2)^2(1-|w|^2)} +  \frac{i dw \wedge d\ol{w}\wedge \frac{1}{2i}\left( \frac{\ol{z}-w}{1-zw}dz - \frac{z-\ol{w}}{1-\ol{zw}}d\ol{z}\right)}{(1-|w|^2)^2 (1-|z|^2)}\right)\\
&= \frac{1}{\sin^2 b} \int_{R \times \pa\D} |\iota_t^* f|^2 \frac{i dz \wedge d\ol{z}\wedge 2(\sin^2 b) d\theta}{(1-|z|^2)^2} \leq 4\pi^2 \sup_X |f|^2 (2g - 2) < \infty 
\end{align*}
where $g$ is the genus of $\Sigma$. Therefore, the LHS should be finite; on the other hand,
\begin{align*}
\int_{\delta^{-1}(\beta, \alpha)} i \pa\opa |f|^2 \wedge \frac{d \delta \wedge d^c \delta}{\delta(1-\delta)} 
& = \int_{\beta}^{\alpha} \frac{dt}{t(1-t)} \int_{M_{t}}i \pa f \wedge \opa \ol{f} \wedge d^c (-\delta),
\end{align*}
and the integrability requires 
\[
\limsup_{t \nearrow \pi/2}  \int_{M_{s}}i \pa f \wedge \opa \ol{f} \wedge d^c (-\delta) = 0.
\]

We can compute this limit in two ways. 
Note that we may apply Fact \ref{Fatou} to not only $f$ but also
\[
\frac{\pa^2 f}{\pa z^2}, \frac{\pa^2 f}{\pa z\pa w}, \frac{\pa^2 f}{\pa w^2}
\]
since they are bounded holomorphic functions on $\D \times \D$ from Cauchy's estimate, and we obtain their boundary value functions on $\D \times \pa \D \sqcup \pa \D \times \D$, which are CR functions. By abuse of notation, we express the boundary value functions by the same symbols. Using the trivialization $\iota_t$ of $M_t$, the bounded convergence theorem yields
\begin{align}
\label{boundary1}
0 &= \lim_{t \nearrow \pi/2}\int_{M_{t}} i \pa f \wedge \opa \ol{f} \wedge d^c (-\delta) \\
&= \lim_{t \nearrow \pi/2} \int_{R \times \pa\D} \iota_t^* \left(i\pa f \wedge \opa \ol{f} \wedge d^c (-\delta)\right) \notag \\ 
&= \int_{R \times \pa\D} \iota_{\pi/2}^* \left(i\pa f \wedge \opa \ol{f} \wedge d^c (-\delta)\right) \notag \\
&= \int_{M} \left|\frac{\pa f}{\pa z}\right|^2 i dz\wedge d\ol{z} \wedge \frac{1-|z|^2}{|1 - ze^{i\varphi}|^2} d\varphi \notag
\end{align}
where we used the coordinate $(z, e^{i\varphi}) \in \D \times \pa\D$ for $\iota_{\pi/2}(R \times \pa\D) \subset M = \D \times \pa\D/\Gamma$.
Using another trivialization $\kappa_t$ of $M_t$, 
\[
\kappa_t \colon \pa\D \times R' \to M_t, \quad (e^{i\theta'}, w) \mapsto \left( \frac{(\sin t)e^{i\theta'}+\ol{w}}{1+w(\sin t)e^{i\theta'}}, w\right) 
\]
for $t \in (0,\pi/2]$ where $R'$ is a fundamental domain of the conjugated action of $\Gamma$ on $\D$, we similarly have
\begin{equation}
\label{boundary2}
0 = \int_{M'} \left|\frac{\pa f}{\pa w}\right|^2 i dw\wedge d\ol{w} \wedge \frac{1-|w|^2}{|1 - we^{i\varphi'}|^2} d\varphi'
\end{equation}
where we used the coordinate $(e^{i\varphi'}, w) \in \pa\D \times \D$ for $\iota'_{\pi/2}(\pa\D \times R') \subset M' = \pa\D \times \D/\Gamma$.

Equations (\ref{boundary1}) and (\ref{boundary2}) imply that the boundary value functions
$f(z, e^{i\varphi})$ and $f(e^{i\varphi'}, w)$ are constant functions in $z$ and $w$ for almost all $e^{i\varphi}$ and $e^{i\varphi'} \in \pa\D$ since these functions are holomorphic in $z$ and $w$ respectively. 
Now it follows that $\tilde{f}(z, w) = f(e^{i\varphi'}, e^{i\varphi})\colon \pa\D \times \pa\D \to \C$ agrees with a constant function almost everywhere, and we finish this proof.
\end{proof}

\begin{Remark}
The integration formula used in the proof is equivalent to 
Demailly's Lelong--Jensen formula \cite{De}. Exploiting this formula, a notion of Hardy space for hyperconvex domains in $\C^n$, \emph{Poletsky--Stessin Hardy spaces}, was introduced in Alan \cite{Al} and Poletsky--Stessin \cite{PSte} independently (cf. Alan--G{\"o}{\u{g}}{\"u}{\c{s}} \cite{AlG}). The proof above actually shows the triviality of $L^2$ Hardy space of $X \subset Y,  Y'$.
\end{Remark}

\begin{Remark}
Yet another proof for the Liouville property which does not employ Fact \ref{Fatou} can be obtained by a method similar to \cite{Ad}, which will be discussed in the author's forthcoming article. 
As in \cite{Ad}, we may show that all the weighted Bergman space of order $> -1$ of $X \subset Y, Y'$ is infinite dimensional in spite of the fact that its $L^2$ Hardy space is trivial.
\end{Remark}

\section{Open problems}
We shall pose two open problems for further study. 

\begin{Problem}
Do other Grauert tubes of finite maximal radius give similar example of hyperconvex manifolds without non-constant bounded holomorphic function? 
\end{Problem}

\begin{Problem}
Is there any domain with Levi-flat boundary having positive Diederich--Forn{\ae}ss index and non-constant bounded holomorphic function? 
\end{Problem}

Problem 2 is a variant of an open problem raised by Sidney Frankel (cf. Ohsawa \cite{O}),
to classify Levi-flat hypersurfaces that bound domains with non-constant bounded holomorphic functions.

\subsection*{Acknowledgement} 
The author is grateful to Kang-Tae Kim, who explained him the notion of parabolic manifold 
when he was a postdoc at SRC-GAIA, which is supported by an NRF grant 2011-0030044 of the Ministry of Education, the Republic of Korea.
This work was also supported by JSPS KAKENHI Grant Number 26800057.

\end{document}